\documentclass[12pt]{amsart}
\usepackage{graphicx} 
\usepackage[utf8]{inputenc}
\usepackage{geometry}
\usepackage{amssymb}
\usepackage{amsthm, amsfonts, empheq}
\usepackage{mathtools}
\usepackage{enumitem}
\usepackage{tikz-cd} 
\usepackage{hyperref}
\usepackage{comment}
\usepackage[capitalize]{cleveref}

\newtheorem{theorem}{Theorem}[section]
\newtheorem{corollary}[theorem]{Corollary}
\newtheorem{lemma}[theorem]{Lemma}
\newtheorem{proposition}[theorem]{Proposition}

\theoremstyle{definition}
\newtheorem{definition}[theorem]{Definition}
\newtheorem{remark}[theorem]{Remark}
\newtheorem{example}[theorem]{Example}

\DeclareMathOperator{\QQ}{\mathbb{Q}}
\DeclareMathOperator{\ZZ}{\mathbb{Z}}

\DeclareMathOperator{\OO}{\mathcal{O}}
\DeclareMathOperator{\FF}{\mathbb{F}}
\DeclareMathOperator{\gal}{Gal}

\DeclareMathOperator{\GL}{GL}

\DeclareMathOperator{\tors}{tors}

\DeclareMathOperator{\D}{D}
\DeclareMathOperator{\Q}{Q}

\begin{document}

\title[Torsion of elliptic curves with rational \(j\)-invariant]{Torsion of elliptic curves over quartic number fields with rational \(j\)-invariant}

\author[Lucas Hamada]{Lucas Hamada}
\address{
Department of Mathematics \\
Institute of Science Tokyo \\
2-12-1, O-okayama, Meguro-ku, Tokyo 152-8551 \\
Japan
}
\email{lucas.h.r.hamada@gmail.com}

\date{}

\begin{abstract}
Let $E$ be an elliptic curve defined over a quartic extension $K$ of $\QQ$ with $j(E)\in\QQ$. We classify the possible group structures of $E(K)_{\tors}$, concluding that the only additional possibilities, besides those occuring over quartic fields by base change from elliptic curves over $\QQ$, are $\ZZ/17\ZZ$ and $\ZZ/21\ZZ$.
\end{abstract}



\maketitle

\section{Introduction}

    Let \(E\) be an elliptic curve defined over a number field \(K\). According to the Mordell-Weil theorem, the Mordell-Weil group \(E(K)\) of \(K\)-rational points of \(E\) is a finitely generated abelian group, which can be expressed as:
    \begin{equation*}
        E(K) \cong E(K)_{\tors} \oplus \mathbb{Z}^{r},
    \end{equation*}
    where \(E(K)_{\tors}\) denotes its torsion subgroup.
    The first major breakthrough regarding the torsion subgroup of an elliptic curve was achieved by B. Mazur, who classified, up to isomorphism, all possible torsion subgroups \(E(\QQ)_{\tors}\) of elliptic curves \(E\) defined over \(\QQ\). This result is presented in the following theorem:
    
    \begin{theorem}[B. Mazur {\cite[Theorem 8]{Mazur_modcur}}]\label{mazur_list}
    Let \(E/\QQ\) be an elliptic curve. Then, \(E(\QQ)_{\tors}\) is isomorphic to one of the following \(15\) groups\(:\)
        \begin{align*}
           & \mathbb{Z}\slash N_1 \mathbb{Z}, &\hspace{1cm}  &  N_1 = 1,\ldots , 10, 12,\\
           & \mathbb{Z}\slash 2\mathbb{Z} \oplus \mathbb{Z} \slash 2N_2 \mathbb{Z},   
            & \hspace{1cm} & N_2 = 1,\ldots , 4.
        \end{align*}
    \end{theorem}

    Several analogues of Theorem \ref{mazur_list} were proved for different classes of elliptic curves:
    Let \(\Phi(d)\) denote the set of all isomorphism classes of torsion subgroup \(E(K)_{\tors}\), where \(K\) runs through all extensions of \(\QQ\) of degree \(d\) and \(E\) runs through all elliptic curves defined over \(K\). The classification of \(\Phi(d)\) has been obtained for several values of \(d\):
    \begin{itemize}
        \item[--] \(\Phi(2)\): By S. Kamienny \cite{kamienny_1, kamienny_2}, and M. Kenku and F. Momose \cite{momose_kenku}.
        \item[--] \(\Phi(3)\): By D. Jeon, C. H. Kim, and A. Schweizer \cite{Jeon2004OnTT}, and  M. Derickx, A. Etropolski, J. Morrow, M. van Hoeij, and D. Zureick-Brown \cite{10.2140/ant.2021.15.1837}.
        \item[--] \(\Phi(4)\): By M. Derickx and F. Najman \cite{DF}.
    \end{itemize}
    Moreover, let \(\Phi_{\QQ}(d)\) denote the set of all isomorphism classes of torsion subgroups \(E(K)_{\tors}\), where \(E\) runs through all elliptic curves defined over \(\QQ\) and \(K\) runs through all extensions of \(\QQ\) of degree \(d\). The following classifications are known:
    \begin{itemize}
        \item[--] \(\Phi_{\QQ}(2)\) and \(\Phi_{\QQ}(3)\): By F. Najman \cite{najman_2016},
        \item[--] \(\Phi_{\QQ}(4)\): By M. Chou \cite{CHOU}, and E. Gonz\'{a}lez-Jim\'{e}nez and F. Najman \cite{GonzalezJimenez2016GrowthOT}
         \item[--] \(\Phi_{\QQ}(p)\) for primes \(p \geq 5\): By Gonz\'{a}lez-Jim\'{e}nez and F. Najman \cite{GonzalezJimenez2016CompleteCO, GonzalezJimenez2016GrowthOT}.
         \item[--] \(\Phi_{\QQ}(6)\): By T. Gu\v{z}vi\'{c}  \cite{guzvic2} and N. Ad\v{z}aga and T. Gu\v{z}vi\'{c} \cite{Adzaga}.
         \item[--] \(\Phi_{\QQ}(d)\), for integers \(d\) whose smallest prime divisor is greater than 7: By E. Gonz\'{a}lez-Jim\'{e}nez \cite{GonzalezJimenez2016GrowthOT}.
         \item[--] \(\Phi_{\QQ}(pq)\), where \(p,q\) are primes with \(pq \neq 4,6\): By T. Gu\v{z}vi\'{c} \cite{guzvic3}.
    \end{itemize}

    In this paper, we study the torsion subgroup structure of elliptic curves with \(j\)-invariant belonging to \(\QQ\). Let \(\Phi_{j \in \QQ}(d)\) denote the set of all isomorphism classes of torsion subgroups \(E(K)_{\tors}\), where \(K\) runs through all extensions of \(\QQ\) of degree \(d\) and \(E\) runs through all elliptic curves defined over \(K\) with \(j(E) \in \QQ\). The first significant work on this problem was conducted by T. Gu\v{z}vi\'{c} \cite{guzvic}, where he determined \(\Phi_{j \in \QQ}(p)\) for all prime numbers \(p\). Following this, J. E. Cremona and F. Najman \cite{Cremona_Najman} extended the results to determine \(\Phi_{j \in \QQ}(d)\) for all \(d\) not divisible by primes less than 11.

   The goal of this paper is to determine the set \(\Phi_{j \in \QQ}(4)\). Specifically, we provide a complete description as follows:

    \begin{theorem}\label{thm: main theorem}
        Let \(E\) be an elliptic curve defined over a quartic extension \(K\) of \(\QQ\), with \(j(E) \in \QQ\). Then, the torsion subgroup \(E(K)_{\tors}\) is isomorphic to one of the following groups\(:\)
    \begin{align*}
        & \ZZ/N_1\ZZ,     & \hspace{1cm}   & N_1 = 1,\; \ldots, 10, 12, 13, 15, 16, 17, 20, 21, 24,\\
    &\ZZ/2\ZZ \oplus \ZZ/2N_2\ZZ, &\hspace{1cm}  & N_2 = 1,\; \ldots, 6, 8,\\
    &\ZZ/3\ZZ \oplus \ZZ/3N_3\ZZ, &\hspace{1cm}  & N_3 = 1,2,\\
    &\ZZ/4\ZZ \oplus \ZZ/4N_4\ZZ, &\hspace{1cm}  & N_4 = 1,2,\\
    &\ZZ/5\ZZ \oplus \ZZ/5\ZZ,\\
    &\ZZ/6\ZZ \oplus \ZZ/6\ZZ.
        \end{align*}
    Moreover, each group in the list above occurs as the torsion subgroup \(E(K)_{\tors}\) for some elliptic curve \(E\) defined over some quartic extension \(K\) of \(\QQ\) with \(j(E) \in \QQ\).
    \end{theorem}

\subsection*{Outline of the paper}

Section \ref{section: On torsion of elliptic curves over quartic fields} recalls the quartic torsion classifications used later. Section \ref{section: generalities on torsion points of rational elliptic curves} collects facts about torsion points on rational elliptic curves. In Section \ref{section: Torsion points of prime-power order} we prove a restriction on points of prime-power order using the theorem of Daniels and Lozano-Robledo. Section \ref{section: lemmas on elliptic curves} discusses the passage from an elliptic curve with rational \(j\)-invariant to a rational quadratic twist. Finally, Section \ref{section: proof of the main theorem} proves Theorem \ref{thm: main theorem} and gives examples for the two torsion structures that do not arise
from the rational-base-change list.

\subsection*{Computational verification}

The Magma \cite{MAGMA} scripts used to verify the explicit computations in Examples \ref{ex:17} and \ref{ex:21}, together with some finite-group checks to complement the proof of Lemma \ref{lem:9}, are available at the author’s GitHub repository

\begin{center}
    \url{https://github.com/Lucas-H-R-Hamada/Torsion-of-elliptic-curves}
\end{center}

\subsection*{AI declaration}
The original manuscript, with a correct proof, was written in 2024 without any use of AI. With the help of ChatGPT, we found a simpler proof to Lemma \ref{lem:9} which resulted in a new section \ref{section: Torsion points of prime-power order} and improved the readability.

\section{On torsion of elliptic curves over quartic fields}\label{section: On torsion of elliptic curves over quartic fields}

We first recall the classifications of $\Phi(4)$ and $\Phi_{\QQ}(4)$ that will be used in the proof of Theorem~\ref{thm: main theorem}.

\begin{theorem}[M. Derickx and F. Najman \cite{DF}]\label{thm:DF}
        Let \(E\) be an elliptic curve defined over a quartic extension \(K\) of \(\QQ\). Then, the torsion subgroup \(E(K)_{\tors}\) is isomorphic to one of the following groups\(:\)
    \begin{align*}
        & \ZZ/N_1\ZZ,     & \hspace{1cm}   & N_1 = 1,\; \ldots, 18, 20, 21, 22, 24,\\
    &\ZZ/2\ZZ \oplus \ZZ/2N_2\ZZ, &\hspace{1cm}  & N_2 = 1,\; \ldots, 9,\\
    &\ZZ/3\ZZ \oplus \ZZ/3N_3\ZZ, &\hspace{1cm}  & N_3 = 1, 2, 3,\\
    &\ZZ/4\ZZ \oplus \ZZ/4N_4\ZZ, &\hspace{1cm}  & N_4 = 1, 2,\\
    &\ZZ/5\ZZ \oplus \ZZ/5\ZZ,\\
    &\ZZ/6\ZZ \oplus \ZZ/6\ZZ.
        \end{align*}
    \end{theorem}

 \begin{theorem}[M. Chou \cite{CHOU}, E. Gonz\'{a}lez-Jim\'{e}nez and F. Najman \cite{GonzalezJimenez2016GrowthOT}]\label{chou_list}
        Let \(E/\QQ\) be an elliptic curve and let \(K\) be a quartic extension of \(\QQ\). Then, the torsion subgroup \(E(K)_{\tors}\) is isomorphic to one of the following groups\(:\)
        \begin{align*}
        & \ZZ/N_1\ZZ,     & \hspace{1cm}   & N_1 = 1,\; \ldots, 10, 12, 13, 15, 16, 20, 24,\\
    &\ZZ/2\ZZ \oplus \ZZ/2N_2\ZZ, &\hspace{1cm}  & N_2 = 1,\; \ldots, 6, 8,\\
    &\ZZ/3\ZZ \oplus \ZZ/3N_3\ZZ, &\hspace{1cm}  & N_3 = 1,2,\\
    &\ZZ/4\ZZ \oplus \ZZ/4N_4\ZZ, &\hspace{1cm}  & N_4 = 1,2,\\
    &\ZZ/5\ZZ \oplus \ZZ/5\ZZ,\\
    &\ZZ/6\ZZ \oplus \ZZ/6\ZZ.
        \end{align*}
    \end{theorem}  
    
 P.~L.~Clark, P.~Corn, A.~Rice, and J.~Stankewicz \cite{clark_corn_rice_stankewicz_2014}, classify the torsion subgroups of elliptic curves with complex multiplication over number fields of degree at most~13. Here, we state their classification in the case of quartic number fields.

\begin{theorem}[P. L. Clark, P. Corn, A. Rice, and J. Stankewicz \cite{clark_corn_rice_stankewicz_2014}]\label{thm:CM}
Let \(E\) be an elliptic curve defined over a quartic extension \(K\) of \(\QQ\), with complex multiplication. Then, the torsion subgroup \(E(K)_{\mathrm{tors}}\) is isomorphic to one of the following groups\(:\)
\begin{align*}
    & \ZZ/N_1\ZZ,          & \hspace{1cm} & N_1 = 1, \ldots, 8, 10, 12, 13, 21, \\
    & \ZZ/2\ZZ \oplus \ZZ/2N_2\ZZ, & \hspace{1cm} & N_2 = 1, \ldots, 5, \\
    & \ZZ/3\ZZ \oplus \ZZ/3N_3\ZZ, & \hspace{1cm} & N_3 = 1, 2, \\
    & \ZZ/4\ZZ \oplus \ZZ/4\ZZ.
\end{align*}
\end{theorem}

\section{Generalities on torsion points of rational elliptic curves}\label{section: generalities on torsion points of rational elliptic curves}

In this section, we recall some properties of torsion points of elliptic curves defined over \(\QQ\). The main result is the list determined by E. Gonz\'{a}lez-Jimen\'{e}z and F. Najman of the possible degrees of the fields of definition of points of prime order. We fix an algebraic closure \(\overline{\QQ}\) and regard all number fields as subfields of it. For a point \(P\) on a rational elliptic curve, \(\QQ(P)\) is the field generated by its coordinates.

\begin{theorem}[Theorem 5.8 of \cite{GonzalezJimenez2016GrowthOT}]\label{exact_degree}
Let \(E/\QQ\) be an elliptic curve, let \(p\) be a prime, and let \(P\) be a point of order \(p\) in \(E\). Then all of the cases in the table below occur for \(p \leq 13\) or \(p = 37\), and they are the only ones possible. Degrees marked with an asterisk occur only for elliptic curves with complex multiplication.
\[\begin{array}{|c|c|}
\hline
p & [\QQ(P):\QQ]\\
\hline
2 & 1,2,3\\
\hline
3 & 1,2,3,4,6,8\\
\hline
5 & 1,2,4,5,8,10,16,20,24\\
\hline
7 & 1,2,3,6,7,9,12,14,18,21,24^{\text{*}},36,42,48\\
\hline
11 & 5,10,20^{\text{*}}, 40^{\text{*}},55, 80^{\text{*}},100^{\text{*}},110,120\\
\hline
13 & 3,4,6,12, 24^{\text{*}},39, 48^{\text{*}},52,72,78,96, 144^{\text{*}},156,168\\
\hline
37 & 12,36, 72^{\text{*}},444, 1296^{\text{*}},1332, 1368\\
\hline
\end{array}\]
For all other \(p\), the possibilities for \([\QQ(P):\QQ]\) are as is given below. The degrees in items \((iii)\)-\((v)\) occur only for CM elliptic curves \(E/\QQ\). Furthermore the degrees in equations \((v)\) occur only for elliptic curves with \(j\)-invariant \(0\). If \cite[Conjecture 3.5]{GonzalezJimenez2016GrowthOT} is true, then the degrees in equations \((vi)\) also occur only for elliptic curves with \(j\)-invariant \(0\). Then
\begin{enumerate}[itemsep=0.3cm]
    \item[\((i)\)] \quad \(p^2 - 1\) \quad   for all \(p\),
    \item[\((ii)\)] \quad \(8,\; 16,\; 32^{\text{*}},\; 136,\; 256^{\text{*}},\; 272,\; 288\) \quad for \(p = 17\),
    \item[\((iii)\)] \quad \((p-1)/2,\;\; p-1,\;\; p(p-1)/2,\;\; p(p-1)\) \quad if \(p \in \{19,43,67, 163\}\),
    \item[\((iv)\)] \quad \(2(p-1),\;\; (p-1)^2\) \quad if \(p \equiv 1 \pmod{3}\) or \(\left(\frac{-D}{p} \right) = 1\) for any \(D \in \mathcal{CM}\),
    \item[\((v)\)] \quad \((p-1)^2/3,\;\; 2(p-1)^2/3\)\quad if \(p \equiv 4,7 \pmod{9}\),
    \item[\((vi)\)]\quad \((p^2 - 1)/3,\;\; 2(p^2 - 1)/3\)\quad if \(p \equiv 2,5 \pmod{9}\),
\end{enumerate}
where \(\mathcal{CM} = \{1,2,7,11,19,43,67,163 \}\).

Apart from the cases above that have been proved to appear, the only other options that might be possible are: 
\begin{enumerate}
    \item[\((vii)\)] \quad \((p^2 - 1)/3,\;\; 2(p^2 - 1)/3\) \quad if \(p \equiv 8 \pmod{9}\).
\end{enumerate}
\end{theorem}

\begin{remark}
    Let \(E/\QQ\) be an elliptic curve \textit{without complex multiplication} and let \(P_{17} \in E(\overline{\QQ})\) have order 17. By Theorem \ref{exact_degree}, using \((i)\), \((ii)\) and \((vii)\), we have
    \begin{equation*}
      [\QQ(P_{17}):\QQ] \in \{ 8,\; 16,\;96,\; 136,\;192,\; 272,\; 288 \}.
    \end{equation*}
\end{remark}

To obtain stronger restrictions on the field of definition of a torsion point, we use cyclic isogenies.

\begin{definition}[Cyclic isogeny]
    Let \(E_1\) and \(E_2\) be two elliptic curves defined over a field \(K\) of characteristic zero.
    \begin{enumerate}
        \item An \emph{isogeny} over \(K\) between \(E_1\) and \(E_2\) is a non-constant morphism \(f: E_1 \to E_2\) defined over \(K\) with \(f(\OO_{E_1}) = \OO_{E_2}\).
        \item The isogeny \(f\) is a \emph{cyclic} \(n\)-\emph{isogeny} if its kernel \(f^{-1}(\OO_{E_2})\) is a cyclic group of order \(n\). In this case, we say that \(E_1\) has a cyclic \(n\)-isogeny over \(K\).
    \end{enumerate}
\end{definition}

\begin{remark}\label{cyclic n-isogeny subgroup}
    By \cite[Proposition III.4.12 and Remark III.4.13.2]{Silverman:1338326}, $E/K$ has a cyclic $n$-isogeny over $K$ if and only if $E(\overline K)$ contains a cyclic subgroup of order \(n\) invariant under $\gal(\overline K/K)$.
\end{remark}

 \begin{lemma}[Lemmas 2.7 and 2.8 of \cite{LH}]\label{Lemma: degree of definition, isogeny}
        Let \(E/\QQ\) be an elliptic curve and \(K\) a finite Galois extension of \(\QQ\).
        \begin{enumerate}
            \item If \(P_n \in E(K)\) has order n and \(\left<P_n\right>\), the group generated by \(P_n\), is \(\gal(K/\QQ)\)-invariant, then \([\QQ(P_n):\QQ]\) divides \(\varphi(n)\), where \(\varphi\) is the Euler function.
            \item If \(E(K)[p^n] = \left<P_{p^n},Q_p\right> \cong \ZZ/p^n\ZZ \oplus \ZZ/p\ZZ \), where \(p\) is a prime number coprime to the degree \([K:\QQ]\), \(P_{p^n} \in E(K)[p^n] - E(\QQ)\) is a point of order \(p^n\) that is not defined in \(\QQ\), and \(Q_p \in E(\QQ)[p]\) is a point of order \(p\), then \([\QQ(P_{p^n}):\QQ]\) divides \(p-1\). 
        \end{enumerate}
    \end{lemma}

    The following is a slight generalization of \cite[Lemma 3.10]{CHOU}.

    \begin{lemma}\label{cyclic_isogeny_when_cn+cnm}
    Let \(K\) be a Galois extension of \(\QQ\), and let \(E/\QQ\) be an elliptic curve such that \(E(K)[nm] \cong \ZZ/nm\ZZ \oplus \ZZ/m\ZZ\) for some positive integers \(n\) and \(m\) with \(n \geq 2\). Then, \(E\) has a cyclic \(n\)-isogeny over \(\QQ\).
    \end{lemma}

    \begin{proof}
        Set \(\Phi \coloneqq [m]\left(E(K)[nm]\right)\). Then $\Phi\cong\ZZ/n\ZZ$. Since $K/\QQ$ is Galois and $E/\QQ$, the group $E(K)[nm]$ is stable under \(\gal(\overline{\QQ}/\QQ)\); hence so is its characteristic image $\Phi$. Remark~\ref{cyclic n-isogeny subgroup} gives a cyclic $n$-isogeny over $\QQ$.
    \end{proof}

\section{Torsion points of prime-power order}\label{section: Torsion points of prime-power order}

 The prime-order degree restriction recalled in Section \ref{section: generalities on torsion points of rational elliptic curves} can be complemented by a restriction on points of order \(p^n\). We use the theorem of Daniels and Lozano-Robledo \cite[Theorem 1.4(1)]{DanielRobledo}, which implies that 
 \begin{equation}\label{(1)}
     \QQ(E[p^2]) \neq \QQ(E[p])
 \end{equation}
 for every elliptic curve \(E/\QQ\) and every odd prime \(p\). For an integer \(r \geq 1\), we write \(E[r]\) for the full \(r\)-torsion subgroup of \(\QQ\), and \(\QQ(E[r])\) for the field generated by the coordinates of its points. The following proposition is a consequence of \eqref{(1)} and the Galois action on \(E[p^n]\).

 \begin{proposition}\label{proposition 4.1}
    Let \(E/\QQ\) be an elliptic curve, let \(p\) be an odd prime, and let \(n \geq 2\). Suppose that \(T \in E(\overline{\QQ})\) has order \(p^n\) and \(p \nmid [\QQ(T):\QQ]\). Then \(\left<p^{n-1}T\right>\) is invariant under \(\gal(\overline{\QQ}/\QQ)\). In particular, \(E\) has a cyclic \(p\)-isogeny over \(\QQ\), and \(\QQ(p^{n-1}T)/\QQ\) is cyclic of degree dividing \(p-1\).
 \end{proposition}

 \begin{proof}
     Put 
     \begin{equation*}
         M = \QQ(E[p^n]),\;\;\; G = \gal(M/\QQ),\;\;\; N = \gal(M/\QQ(E[p])).
     \end{equation*}
     The kernel of \(\GL_2(\ZZ/p^n\ZZ) \to \GL_2(\FF_p)\) has order \(p^{4(n-1)}\). Hence \(N\) is a normal \(p\)-subgroup of \(G\). Let \(G_T\) be the stabilizer of \(T\). Since \(N\) is normal, \(NG_T\) is a subgroup of \(G\), and 
     \begin{equation*}
         |N\cdot T| = [NG_T : G_T] \mid [G:G_T] = [\QQ(T):\QQ].
     \end{equation*}
     The orbit size is a power of \(p\) and the degree on the right is coprime to \(p\), so \(N\) fixes \(T\). Normality then implies that \(N\) fixes every conjugate \(\sigma T\), since \(\sigma^{-1}N\sigma = N\) for \(\sigma \in G\). 

     Set \(R = p^{n-1}T\). If \(\left<R\right>\) were not Galois invariant, there would be \(\sigma \in G\) for which \(R\) and \(\sigma R\) are independent in \(E[p]\). Multiplication by \(p^{n-1}\) induces an isomorphism \(E[p^n]/pE[p^n] \to E[p]\). Consequently, \(T\) and \(\sigma T\) form a \(\ZZ/p^n\ZZ\)-basis of \(E[p^n]\). Both are fixed by \(N\), so \(N\) acts trivially on \(E[p^n]\) and must be the trivial group. This would imply
     \begin{equation*}
         \QQ(E[p^n]) = \QQ(E[p]),
     \end{equation*}
     and hence \(\QQ(E[p^2]) = \QQ(E[p])\), since \(n \geq 2\). This contradicts \eqref{(1)}. Thus \(\left< R\right>\) is Galois invariant. 

     Remark \ref{cyclic n-isogeny subgroup} gives the cyclic \(p\)-isogeny. The action on \(\left< R \right>\) defines a character into \((\ZZ/p\ZZ)^\times\) whose kernel fixes \(R\). Hence \(\QQ(R)/\QQ\) is cyclic of degree dividing \(p-1\), ans in Lemma \ref{Lemma: degree of definition, isogeny}(1).
 \end{proof}

 \begin{corollary}\label{corollary 4.2}
     Let \(E/\QQ\) be an elliptic curve, and let \(T \in E(\overline{\QQ})\) have order \(9\). If \(3 \nmid [\QQ(T):\QQ]\), then 
     \begin{equation*}
         [\QQ([3]T):\QQ] \leq 2.
     \end{equation*}
 \end{corollary}

 \begin{proof}
     Apply Proposition \ref{proposition 4.1} with \(p = 3\) and \(n = 2\).
 \end{proof}

\section{Elliptic curves with rational $j$-invariant}\label{section: lemmas on elliptic curves}

Let $K$ be a field of characteristic zero and let
\[
E:y^2=x^3+Ax+B,\qquad A,B\in K,
\]
with $j(E)\in\QQ\setminus\{0,1728\}$. Define an elliptic curve over $\QQ$ by
\begin{equation}\label{eq: rational quadratic twist}
E':y^2=x^3-\frac{27j(E)}{j(E)-1728}x
+\frac{54j(E)}{j(E)-1728}.
\end{equation}
Write
\[
A'=-\frac{27j(E)}{j(E)-1728},\qquad
B'=\frac{54j(E)}{j(E)-1728}.
\]
A direct calculation gives $j(E')=j(E)$. Put
\[
D:=\frac{AB'}{A'B} \in K^\times.
\]
By the proof of \cite[Proposition III.1.4(b)]{Silverman:1338326}, there is an isomorphism
\begin{equation}\label{eq: isomorphism_twist}
\begin{aligned}
\phi:E&\xrightarrow{\sim}E',\\
(x,y)&\longmapsto(Dx,D\sqrt D\,y),
\end{aligned}
\end{equation}
defined over $K(\sqrt D)$, and
\begin{equation}\label{eq: coefficients twist}
A'=AD^2,\qquad B'=BD^3.
\end{equation}
Thus $E'$ is the quadratic twist of $E$ by $D$ in this convention.

\begin{remark}\label{rmk:twisted points}
Assume $\sqrt D\notin K$ and let $P=(x,y)\in E(K)$. The point \(\phi(P)=(Dx,D\sqrt D\,y)\) belongs to \(E'(K)\) if and only if \(y = 0\). Including the identity point, this says that \(\phi(P) \in E'(K)\) if and only if \(P \in E(K)[2]\). Consequently, for every \(P \in E(K)\setminus E(K)[2]\), 
\begin{equation*}
    K(\phi(P)) = K(\sqrt D).
\end{equation*}
\end{remark}

\begin{theorem}\label{proposition_principal}
Let $K,E,E',D,$ and $\phi$ be as above. If $K(\sqrt D)=K(\sqrt m)$ for some $m\in K^\times$, then there is an elliptic curve $\widetilde E/\QQ(m)$ that is $K$-isomorphic to $E$.
\end{theorem}

\begin{proof}
If $K(\sqrt D)=K$, then $E$ is already $K$-isomorphic to the rational curve
$E'$, so there is nothing to prove. Assume the extension is nontrivial. Since two nontrivial quadratic extensions $K(\sqrt D)$ and $K(\sqrt m)$ coincide, we have $D/m\in K^{\times2}$. For completeness, one can also see this from $\sqrt D=a+b\sqrt m$ with $a,b\in K$: if $a\ne0$, then squaring after subtracting $a$ shows that $\sqrt D\in K$, contrary to the nontriviality of the extension. Hence $a=0$ and $D=b^2m$ for some $b\in K^{\times}$.

Using \eqref{eq: coefficients twist},
\[
Ab^4=\frac{A'}{m^2}\in\QQ(m),\qquad
Bb^6=\frac{B'}{m^3}\in\QQ(m).
\]
Therefore
\[
\widetilde E:y^2=x^3+Ab^4x+Bb^6
\]
is defined over $\QQ(m)$. The change of variables $(x,y)\mapsto(b^2x,b^3y)$ gives a $K$-isomorphism $E\simeq\widetilde E$.
\end{proof}

We now apply Corollary \ref{corollary 4.2} to a quadratic twist. The argument does not require the field to be quartic or Galois.

\begin{theorem}\label{theorem 5.3}
    Let \(K\) be a number field with \(3 \nmid [K:\QQ]\), and let \(E/K\) satisfy \(j(E) \in \QQ\setminus \{0,1728\}\). If \(E(K)\) contains a point of order 9, then \(E\) is \(K\)-isomorphic to an elliptic curve defined over \(\QQ\). More precisely, with \(E', D\), and \(\phi\) as above, there is \(m \in \QQ^\times\) such that 
    \begin{equation*}
        K(\sqrt{D}) = K(\sqrt{m})
    \end{equation*}
\end{theorem}

\begin{proof}
    Set \(L = K(\sqrt{D})\). If \(L = K\), take \(m = 1\) and use the \(K\)-isomorphism \(\phi\colon E\to E'\). Otherwise \([L:\QQ] = 2[K:\QQ]\) is coprime to 3. Let \(P \in E(K)\) have order 9, and put 
    \begin{equation*}
        T = \phi(P) \in E'(L),\;\;\; R = [3]T = \phi([3]P).
    \end{equation*}
    Since \(\QQ(T) \subseteq L\), Corollary \ref{corollary 4.2} gives \([\QQ(R) : \QQ] \leq 2\). On the other hand \([3]P\) has order 3, so Remark \ref{rmk:twisted points} gives \(L = K(R)\). In particular, \(R \not\in E'(K)\), and \(\QQ(R)\) is a quadratic field \(\QQ(\sqrt{m})\) for some \(m \in \QQ^\times\). Therefore
    \begin{equation*}
        L = K(R) = K(\sqrt{m}).
    \end{equation*}
    Theorem \ref{proposition_principal} gives a rational model \(K\)-isomorphic to \(E\).
    \end{proof}

\section{Proof of Theorem \ref{thm: main theorem}}\label{section: proof of the main theorem}

Let $K/\QQ$ be quartic and let $E/K$ have $j(E)\in\QQ$. By
Theorem~\ref{thm:CM}, none of the groups that will be excluded below occurs
in the CM case. We may therefore assume that $E$ is non-CM; in particular,
$j(E)\notin\{0,1728\}$. Let $E'/\QQ$, $D$, and
\[
L:=K(\sqrt D)
\]
be as in Section~\ref{section: lemmas on elliptic curves}. If $L=K$, then $E$ is already
$K$-isomorphic to $E'/\QQ$, so $E(K)_{\tors}$ belongs to the list
$\Phi_{\QQ}(4)$ of Theorem~\ref{chou_list}. Hence, for the exclusion argument,
we may and do assume
\begin{equation}\label{eq:Ldegree}
[L:K]=2,\qquad [L:\QQ]=8.
\end{equation}

Comparing Theorems~\ref{thm:DF} and \ref{thm: main theorem}, it is enough to exclude
\begin{equation}\label{eq:forbidden}
\begin{array}{ll}
\ZZ/N_1\ZZ, & N_1=11,14,18,22,\\[1mm]
\ZZ/2\ZZ\oplus\ZZ/2N_2\ZZ, & N_2=7,9,\\[1mm]
\ZZ/3\ZZ\oplus\ZZ/9\ZZ. &
\end{array}
\end{equation}
We do so in three lemmas.

    \begin{lemma}\label{lem:no11}
        \(E(K)_{\tors}\) has no points of order \(11\).
    \end{lemma}

    \begin{proof}
        Suppose \(P_{11} \in E(K)_{\tors}\) is a point of order 11, and let \(\phi(P_{11})\) be its image in \(E'(L)_{\tors}\). By Theorem \ref{exact_degree}, we know that \([\QQ(\phi(P_{11})):\QQ]\) does not divide 8, hence it is impossible that \(\QQ(\phi(P_{11})) \subseteq L\).
    \end{proof}

    \begin{lemma}\label{lem:7}
        If \(E(K)_{\tors}\) has a point of order \(7\), then \(L = K(\sqrt{m})\) for some \(m \in \QQ\).
    \end{lemma}

    \begin{proof}
Let \(P_7\in E(K)_{\mathrm{tors}}\) be a point of order \(7\), and put \(T:=\varphi(P_7)\in E'(L)[7]\). Since \(T\in E'(L)\), we have \(\QQ(T) \subseteq L\). Therefore,
\begin{equation*}
    [\QQ(T):\QQ] \mid [L:\QQ] = 8.
\end{equation*}
On the other hand, Theorem~\ref{exact_degree} gives the possible degrees of the field of definition of a point of order \(7\) on an elliptic curve defined over \(\QQ\). Among those possibilities, the only degrees dividing \(8\) are \(1\) and \(2\). Hence \([\QQ(T):\QQ] \in \{1,2\}\). By the standing assumption \([L:K]=2\) and Remark \ref{rmk:twisted points}, the degree-one case is impossible and there exists \(m\in\QQ^{\times} \setminus (\QQ^\times)^2\) such that \(\QQ(T) = \QQ(\sqrt{m})\). Since \(P_7\) is not a point of order \(2\), Remark~\ref{rmk:twisted points} applies and yields 
\begin{equation*}
        L = K(\phi(P_7)) = K\QQ(\phi(P_7)) = K\QQ(\sqrt{m}) = K(\sqrt{m}).
\end{equation*}
\end{proof}

    \begin{lemma}\label{lem:9}
        If \(E(K)_{\tors}\) has a point of order \(9\), then \(L = K(\sqrt{m})\) for some \(m \in \QQ\).
    \end{lemma}

    \begin{proof}
        Apply Theorem \ref{theorem 5.3}, since \([K:\QQ] = 4\) is coprime to 3.
    \end{proof}

We can now finish the exclusion. Lemma~\ref{lem:no11} excludes $\ZZ/11\ZZ$ and $\ZZ/22\ZZ$. If a remaining group in \eqref{eq:forbidden} contains a point of order $7$, Lemma~\ref{lem:7} and Theorem~\ref{proposition_principal} show that $E$ is $K$-isomorphic to an elliptic curve over $\QQ$, so its torsion subgroup must lie in $\Phi_{\QQ}(4)$, contradicting Theorem~\ref{chou_list}. The same argument, using Lemma~\ref{lem:9}, excludes the remaining groups containing a point of order $9$. Thus no groups in \eqref{eq:forbidden} belongs to $\Phi_{j\in\QQ}(4)$.

\subsection{Examples of elliptic curves for each torsion structure}

It remains to prove existence. Every group in Theorem~\ref{thm: main theorem} except $\ZZ/17\ZZ$ and $\ZZ/21\ZZ$ occurs as $E(K)_{\tors}$ for some elliptic curve $E/\QQ$ and some quartic extension $K/\QQ$, by
Theorem~\ref{chou_list}. We give examples for the two additional groups.

\begin{example}[A curve with $17$-torsion; D. Jeon \cite{2023arXiv230505851J}]\label{ex:17}
Let
\[
K=\QQ(\alpha),\qquad
2\alpha^4-5\alpha^3-7\alpha^2+10\alpha+8=0.
\]
This is a cyclic quartic field. Define
\begin{align*}
b={}&-\frac{1337}{32}\alpha^3+\frac{3233}{64}\alpha^2
+\frac{12285}{64}\alpha+\frac{1135}{16},\\
c={}&-\frac{59}{16}\alpha^3+\frac{147}{32}\alpha^2
+\frac{535}{32}\alpha+\frac{45}{8}.
\end{align*}
Consider the following elliptic curve in Tate normal form
\begin{equation}\label{eq:jeon17}
E=E(b,c):\qquad
y^2+(1-c)xy-by=x^3-bx^2.
\end{equation}
Then $(0,0)$ has exact order $17$, and a direct computation of the \(j\)-invariant gives
\[
j(E)=-\frac{882216989}{131072}\in\QQ.
\]
Consequently $\ZZ/17\ZZ\subseteq E(K)_{\tors}$. By the quartic classification in Theorem~\ref{thm:DF}, no larger torsion group containing a point of order $17$ can occur; hence
\[
E(K)_{\tors}\cong\ZZ/17\ZZ.
\]
The formulas above are those of \cite[formula (4.9)]{2023arXiv230505851J}; the calculation can be verified using the MAGMA script cited in the computational-verification note in the Introduction.
\end{example}

\begin{example}[A curve with $21$-torsion; Clark--Corn--Rice--Stankewicz
\cite{clark_corn_rice_stankewicz_2014}]\label{ex:21}
Let
\[
K=\Q(\alpha),\qquad \alpha^4-\alpha^3+2\alpha+1=0,
\]
and consider
\[
E:y^2=x^3-
\bigl(371952\alpha^3+3373488\alpha^2+3777840\alpha+1228608\bigr).
\]
Then $j(E)=0$ and
\[
E(K)_{\tors}\cong\ZZ/21\ZZ.
\]
The calculation can also be verified using the MAGMA script cited in the computational-verification note in the Introduction.
\end{example}

Examples~\ref{ex:17} and \ref{ex:21}, together with
Theorem~\ref{chou_list}, complete the proof of Theorem~\ref{thm: main theorem}.

\subsection*{Acknowledgments}

The author is deeply grateful to Yuichiro Taguchi and Yuto Nagashima for their insightful discussions and valuable comments on an earlier version of this paper. The author also thanks the anonymous referee for a careful reading and detailed suggestions, which significantly improved the presentation of the paper.


\end{document}